\documentclass[a4paper,12pt]{amsart}
\usepackage{etex}
\usepackage{tikz, tikz-cd}

\usepackage{bm}
\usepackage{amssymb}
\usepackage{mathrsfs}
\usepackage{amsmath, amssymb,amsthm,latexsym,amscd,mathrsfs}
\usepackage{indentfirst}
\usepackage{stmaryrd}
\usepackage{graphicx}
\usepackage{subfigure}
\usepackage{extarrows}
\usepackage{amssymb}
\usepackage{amsmath,amssymb,amsthm,latexsym,amscd,mathrsfs}
\usepackage{indentfirst}
\usepackage{multirow}
\usepackage{latexsym}
\usepackage{amsfonts}
\usepackage{color}
\usepackage{pictexwd,dcpic}
\usepackage{graphicx}
\usepackage{psfrag}
\usepackage{hyperref}
\usepackage{comment}
\usepackage{cases}
\usepackage{dutchcal}

\makeatletter
\@namedef{subjclassname@2020}{{\rm2020} Mathematics Subject Classification}
\makeatother

 \DeclareMathOperator{\Id}{Id}

 \newcommand{\R}{\mathbb{R}}

 \newcommand{\ROM}[1]{\mathrm{\uppercase\expandafter{\romannumeral#1}}}
  \theoremstyle{definition}
   \numberwithin{equation}{section} \theoremstyle{plain}

 \newtheorem{thm}{Theorem}[section]
 \newtheorem{lem}{Lemma}[section]
 
 \newtheorem{cor}{Corollary}[section]
  
 \newtheorem{rem}{Remark}[section]
 \newtheorem{prop}{Proposition}[section]

  \makeatletter

  \numberwithin{equation}{section}
  \allowdisplaybreaks
\makeatother
\title[Homogeneity for $(g, m)=(6,2)$]{\textbf{A concise proof of homogeneity of an isoparametric hypersurface with} $(g, m)=(6, 2)$}
\author[Chao Qian]{Chao Qian}\address{School of Mathematics and Statistics, Beijing Institute of Technology, Beijing
100081, P.R. China}
\email{6120150035@bit.edu.cn}
\author[Z. Z. Tang]{Zizhou Tang}\address{Chern Institute of Mathematics $\&$ LPMC, Nankai University, Tianjin 300071, P. R. China}
\email{zztang@nankai.edu.cn}
\thanks {The project is partially supported by the NSFC (No. 12371048, 12526205), and Nankai Zhide Foundation. %, and the Fundamental Research Funds for the Central Universities (No. 2233300002)
}


\begin{document}

\maketitle

\begin{abstract}
A well-known hypersurface classification theorem states that any isoparametric hypersurface in $S^{13}$ with six principal curvatures is homogeneous. This landmark result was first proved in her Annals paper in 2013 by R. Miyaoka with 58 pages, and with errata in Annals in 2016 with 15 pages. The main purpose of this paper is to provide a concise proof of this theorem through an interplay between topological and geometric insights into the Euler class of an oriented vector bundle.
\end{abstract}

\section{\textbf{Introduction}}\label{sec1}
An isoparametric hypersurface in the unit sphere is a hypersurface with
constant principal curvatures (See \cite{Th00, CR15}). It belongs to a parallel family whose two
singular leaves are the focal submanifolds.  A fundamental theorem of
M\"{u}nzner \cite{Mu80} shows that the number $g$ of distinct principal curvatures can
only be $1,2,3,4$, or $6$, that the multiplicities satisfy
$m_i=m_{i+2}$ (indices modulo $g$), and that the family is described by a
Cartan--M\"{u}nzner polynomial. The cases $g=1,2$ are the
standard hyperspheres and generalized Clifford tori, while Cartan
classified the case $g=3$ \cite{Ca39}.  For $g=4$, the classification theorem \cite{CCJ07,Imm08,Chi11,Chi13,Chi20} states
that every one is of OT-FKM type or belongs to one of the exceptional homogeneous families
with multiplicities $(2,2)$ and $(4,5)$.

For $g=6$, M\"{u}nzner \cite{Mu80}
pointed out that the two multiplicities coincide,
and furthermore, Abresch \cite{Ab83b} proved that their common value must be $m=1$ or $2$.
Dorfmeister and Neher proved the homogeneity of the case $m=1$
\cite{DN85}, and Miyaoka \cite{Mi13} established the classification theorem in the most complicated case of $m=2$, together with a
subsequent erratum \cite{Mi16}. The purpose of the present paper is
to give a short proof of this homogeneity result with quite different method.

\begin{thm}\label{thm:main}
Let $M^{12}\subset S^{13}$ be a connected closed isoparametric
hypersurface with $g=6$. Then $M$ is extrinsically homogeneous.
\end{thm}

The proof combines a global Euler-class calculation with a local
equality-case rigidity argument for the focal shape pencil. Using
the resulting canonical form of the focal shape pencil together
with the Codazzi and Gauss equations, we construct an Ambrose--Singer
connection and thereby establish homogeneity.

%----------------------------------------------------------------------------------------------------------------------------------------------------------------------------------------
\section{Euler class of principal distributions on curvature spheres}\label{sec2}

Let $M^{12}\subset S^{13}$ be a closed isoparametric hypersurface with six distinct principal curvatures. According to M\"{u}nzner \cite{Mu80} and Abresch \cite{Ab83b}, all six principal curvatures have multiplicity $2$. Let $N_+$ and $N_-$ be its two focal submanifolds which are both of dimension 10. For $p\in N_\pm$,
write
\[
 \Sigma_p=S(\nu_p^{S^{13}}N_\pm)\cong S^2
\]
for the unit normal sphere.

Let $B_\eta$ denote the shape operator of
$N_\pm\subset S^{13}$ at $p$ in the normal direction $\eta$. For each
eigenvalue $\mu$, define the $\mu$-eigenbundle as the smooth rank-two bundle(See Part (i) of Theorem 3.1)
\[
 \mathcal E_\mu(p)\longrightarrow\Sigma_p,
 \qquad
 \mathcal E_\mu(p)_\eta=\ker(B_\eta-\mu\Id).
\]
After choosing orientations of $\mathcal E_\mu$ and $\Sigma_p$, its Euler number is
\[
 \chi_\mu(p)=\left\langle
 e\bigl(\mathcal E_\mu(p)\bigr),[\Sigma_p]
 \right\rangle,
\]
where $ e\bigl(\mathcal E_\mu(p)\bigr)$ is the Euler class of $\mathcal E_\mu$, and $[\Sigma_p]$ is the fundamental class of $\Sigma_p$.
\begin{prop}\label{prop:focal-euler-numbers}
After interchanging $N_+$ and $N_-$ if necessary,, the absolute Euler
numbers are
\[
 \boxed{
 \begin{array}{c|ccccc}
 \mu&\sqrt{3}&1/\sqrt{3}&0&-1/\sqrt{3}&-\sqrt{3}\\ \hline
 p\in N_+&3&1&0&1&3\\
 p\in N_-&1&1&0&1&1
 \end{array}}
\]
at every point of the indicated focal submanifold.  In particular,
every such isoparametric family has exactly one focal submanifold for which the
$\sqrt{3}$-eigenbundle has absolute Euler number $3$ at every point.
\end{prop}
\begin{proof}
We will reduce this assertion to integral Euler pairings on the curvature
spheres. According to M\"{u}nzner \cite{Mu80}, $M$ is simply connected and orientable. Choose a global unit normal $\xi$ on $M$ and write
\[
 \lambda_i=\cot\theta_i,
 \qquad
 \theta_i=\theta_1+\frac{(i-1)\pi}{6},
 \qquad
 0<\theta_1<\frac{\pi}{6}.
\]
Let $D_i$ be the corresponding rank-two principal-curvature
distribution.  Each leaf of $D_i$ is a curvature $2$-sphere, and the
focal map
\[
 F_i:M\longrightarrow N_i:=F_i(M),
 \qquad
 F_i(x)=\cos\theta_i\,x+\sin\theta_i\,\xi(x),
\]
is a sphere bundle whose fibers are precisely these leaves. The images
$N_i$ alternate between the two focal submanifolds $N_{\pm}$.

Since $M$ is simply connected, every $D_i$ is orientable.  Choose
orientations and define
\[
 c_i=e(D_i)\in H^2(M;\mathbb Z).
\]
Orient every $D_i$-leaf $L_i$ by $D_i|_{L_i}=TL_i$, and let
\[
 b_i=[L_i]\in H_2(M;\mathbb Z).
\]
Fibers of the same oriented sphere bundle are homologous, so $b_i$ does
not depend on the chosen leaf. Define
\begin{equation}\label{eq:euler-pairings}
 a_{ij}:=\langle c_j, b_i\rangle,
\end{equation}
where
\[
\langle\cdot,\cdot\rangle:
H^2(M;\mathbb Z)\times H_2(M;\mathbb Z)\longrightarrow \mathbb Z
\]
denotes the Kronecker pairing.
The diagonal entries are
\begin{equation}\label{eq:euler-pairing-diagonal}
 a_{ii}=\langle e(TL_i),[L_i]\rangle=\chi(S^2)=2.
\end{equation}

%Then, orient $M$ by the decomposition
%\[
% TM=D_1\oplus\cdots\oplus D_6.
%\]
%The product formula for Euler classes gives
%\[
% e(TM)=c_1c_2c_3c_4c_5c_6.
%\]
%Its evaluation on $[M]$ is $\chi(M)=12$, and hence
%\begin{equation}\label{eq:euler-product-nonzero}
% \boxed{c_1c_2c_3c_4c_5c_6\ne0.}
%\end{equation}
%Here we use the standard identification of the tangent-bundle Euler
%number with the Euler characteristic \cite{MS}.

It remains to relate the focal eigenbundles to the pairings
\eqref{eq:euler-pairings}. Fix $i$ and $p\in N_i$, and define
\[
 \eta_i(x)=-\sin\theta_i\,x+\cos\theta_i\,\xi(x).
\]
For $x\in F_i^{-1}(p)$, this is a unit normal vector to $N_i$ at $p$.
The map
\[
 \eta_i:F_i^{-1}(p)\longrightarrow\Sigma_p
\]
is a diffeomorphism, whose inverse is
\[
 x=\cos\theta_i\,p-\sin\theta_i\,\eta.
\]
Take $X\in D_j(x)$, where $j\ne i$.  Since
$d\xi(X)=-\lambda_jX$, we have
\begin{align}
 dF_i(X)
 &=\bigl(\cos\theta_i-\lambda_j\sin\theta_i\bigr)X \notag\\
 &=\frac{\sin(\theta_j-\theta_i)}{\sin\theta_j}X,
 \label{eq:differential-focal-map}
\end{align}
and
\begin{align}
 d\eta_i(X)
 &=-\bigl(\sin\theta_i+\lambda_j\cos\theta_i\bigr)X \notag\\
 &=-\frac{\cos(\theta_j-\theta_i)}{\sin\theta_j}X.
 \label{eq:differential-normal-map}
\end{align}
The coefficient in \eqref{eq:differential-focal-map} is nonzero.  Since
$dF_i(X)\in T_pN_i$, it follows that the ambient vector $X$ lies in
$T_pN_i$; hence \eqref{eq:differential-normal-map} is tangent to $N_i$
as well.  The Weingarten formula then yields
\begin{equation}\label{eq:focal-weingarten}
 B_{\eta_i(x)}dF_i(X)
 =-d\eta_i(X)
 =\cot(\theta_j-\theta_i)\,dF_i(X).
\end{equation}
Consequently,
\begin{equation}\label{eq:eigenbundle-pullback}
 dF_i:D_j|_{F_i^{-1}(p)}\xrightarrow{\ \cong\ }
 \eta_i^*\mathcal E_{\cot(\theta_j-\theta_i)}(p).
\end{equation}

As $j-i$ runs through $1,\ldots,5$ modulo six, the eigenvalues are
\[
 \cot\frac{\pi}{6},\quad
 \cot\frac{2\pi}{6},\quad
 \cot\frac{3\pi}{6},\quad
 \cot\frac{4\pi}{6},\quad
 \cot\frac{5\pi}{6}.
\]

Each eigenspace has dimension two.  The eigenbundles are smooth, since their spectral projections are polynomials in the
smoothly varying operator $B_\eta$. By naturality of the integral Euler
class, \eqref{eq:eigenbundle-pullback} implies
\begin{equation}\label{eq:focal-euler-pairing}
 \left|\chi_{\cot((j-i)\pi/6)}(p)\right|=|a_{ij}|.
\end{equation}
Thus, determining the Euler numbers of the focal eigenbundles reduces precisely to computing the integral pairings $a_{ij}$. Finally, the proof follows from the computation at P. 78--79 of \cite{Ab83a}.
\end{proof}
\begin{rem}
By Proposition~2.1 and the Euler-number computation for Siffert's example in the Appendix of \cite{Si17}, her example cannot arise from the shape-operator pencil of a focal submanifold of an isoparametric hypersurface in $S^{13}$ (See also \cite{Mi25}).
\end{rem}
%----------------------------------------------------------------------------------------------------------------
\section{Euler number inequality and rigidity}\label{sec3}
This section prepares a real algebraic statement that will be applied to
the normal shape pencil of a focal submanifold. The Euler number equality
needed for rigidity comes from Proposition~\ref{prop:focal-euler-numbers}.

Let $(V, \langle\cdot,\cdot\rangle)$ be a ten-dimensional Euclidean space and let
\[
 B:\R^3\longrightarrow\operatorname{Sym}(V)
\]
be linear.  Suppose that, for every unit vector $n\in S^2$, the self-adjoint operator $B_n$ has eigenvalues given by
\[
 \operatorname{Spec}(B_n)
 =\left\{\sqrt3,\frac1{\sqrt3},0,-\frac1{\sqrt3},-\sqrt3\right\},
\]
each eigenvalue having multiplicity two.  Put
\[
 a=\sqrt3,\qquad b=\frac1{\sqrt3},\qquad
 c=\frac2{\sqrt3},\qquad
 \Lambda=\{a,b,0,-b,-a\},
\]
and let
\[
 \mathcal E_\lambda(n)=\ker(B_n-\lambda\Id),
 \qquad \lambda\in\Lambda,
\]
be the corresponding eigenbundles over $S^2$.  After choosing
orientations, write
\[
 d_\lambda=\left\langle e(\mathcal E_\lambda),[S^2]\right\rangle.
\]
Then the main result is given as follows.
\begin{thm}
\label{thm:extremal-pencil}
\begin{enumerate}
\item Each $\mathcal E_\lambda$ is a smooth orientable real
rank-two bundle, and
\[
 S^2\times V=\bigoplus_{\lambda\in\Lambda}\mathcal E_\lambda,~~
 |d_{-\lambda}|=|d_\lambda|~(\lambda=a,b).
 %\qquad d_0=0.
\]
%With compatible orientations one may arrange
%$d_{-\lambda}=-d_\lambda$ for $\lambda=a,b$.

\item The two extreme eigenbundles satisfy the sharp estimate
\[
 |d_a|=|d_{-a}|\le3.
\]

\item Equality $|d_a|=3$ holds if and only if there are an orthonormal
basis $n_1,n_2,n_3$ of $\R^3$ and an orthogonal decomposition
\[
 V=K\oplus W_a\oplus W_b\oplus W_{-b}\oplus W_{-a},
 \qquad \dim K=\dim W_\lambda=2,
\]
for which
\[
 K=\bigcap_{\xi\in\R^3}\ker B_\xi
\]
and the following real block formula holds.  Set
\[
 I=I_2,\qquad
 J=\begin{pmatrix}0&-1\\1&0\end{pmatrix},
 \qquad Q(x,y)=xI+yJ.
\]
Choose suitable orthonormal bases in the four real planes.  Then
\[
 B_{xn_1+yn_2+zn_3}|_K=0,
\]
and, relative to the ordered decomposition
$W_a\oplus W_b\oplus W_{-b}\oplus W_{-a}$,
\begin{equation}\label{eq:canonical-extremal-pencil}
 B_{xn_1+yn_2+zn_3}|_{K^\perp}
 =
 \begin{pmatrix}
 azI&Q(x,y)&0&0\\
 Q(x,y)^t&bzI&cQ(x,y)&0\\
 0&cQ(x,y)^t&-bzI&Q(x,y)\\
 0&0&Q(x,y)^t&-azI
 \end{pmatrix}.
\end{equation}
Here every entry in this display is a real $2$-by-$2$ block.
%In the equality
%case,
%\[
%  (|d_a|,|d_b|,|d_0|,|d_{-b}|,|d_{-a}|)
% =(3,1,0,1,3).
%\]
%With compatible real orientations, the signed row is
%\[
% \varepsilon(3,1,0,-1,-3),\qquad \varepsilon\in\{1,-1\}.
%\]
\end{enumerate}
\end{thm}

\begin{proof}
We divide the proof into seven steps.

\smallskip
\noindent\textit{1). Eigenbundles and their elementary topology.}
For $\lambda\in\Lambda$, the spectral projection is
\begin{equation}\label{eq:spectral-projector}
 P_\lambda(n)=
 \prod_{\substack{\mu\in\Lambda\\ \mu\ne\lambda}}
 \frac{B_n-\mu\Id}{\lambda-\mu}.
\end{equation}
It is a polynomial in $B_n$, hence depends smoothly on $n$, and it has
constant rank two.  This proves that $\mathcal E_\lambda$ is a smooth
Euclidean plane bundle.  It is orientable because
$H^1(S^2;\mathbb Z_2)=0$.

Let $\alpha:S^2\to S^2$ be the antipodal map.  Since
$B_{-n}=-B_n$,
\[
 \mathcal E_{-\lambda}\simeq\alpha^*\mathcal E_\lambda.
\]
The degree of $\alpha$ is $-1$, which gives
$|d_{-\lambda}|=|d_\lambda|$ and, with compatible orientation choices,
$d_{-\lambda}=-d_\lambda$ for $\lambda=a,b$.

The proof of part (i) is complete.
%We also record why $d_0=0$.  Let
%$q:S^2\to\mathbb{RP}^2$ be the antipodal quotient.  For
%$\lambda=a,b$, the rank-four bundle
%\[
% \overline{\mathcal F}_\lambda([n])
% =\mathcal E_\lambda(n)\oplus\mathcal E_{-\lambda}(n)
%\]
%is well defined over $\mathbb{RP}^2$.  Choose a smooth orientation
%$o(n)$ of $\mathcal E_\lambda(n)$ and orient the pullback of
%$\overline{\mathcal F}_\lambda$ by $o(n)\wedge o(-n)$.  The deck
%transformation exchanges two two-dimensional factors, so it preserves
%this orientation.  Thus $\overline{\mathcal F}_a$ and
%$\overline{\mathcal F}_b$ are oriented.  A fixed orientation of $V$
%then orients their rank-two complement
%$\overline{\mathcal E}_0$, and
%$\mathcal E_0=q^*\overline{\mathcal E}_0$.  Since
%\[
% H^2(\mathbb{RP}^2;\mathbb Z)\simeq\mathbb Z_2,
% \qquad H^2(S^2;\mathbb Z)\simeq\mathbb Z,
%\]
%the homomorphism $q^*$ in degree two is zero.  Naturality of the Euler
%class gives $e(\mathcal E_0)=0$.

\smallskip
\noindent\textit{2). Two spectral variation identities.}
Fix orthonormal vectors $n,u\in\R^3$ and put
\[
 n(t)=\cos t\,n+\sin t\,u,\qquad P_\mu(t)=P_\mu(n(t)).
\]
Linearity gives
\[
 B_{n(t)}=\cos t\,B_n+\sin t\,B_u,\qquad
 B_{n(t)}'(0)=B_u,\qquad B_{n(t)}''(0)=-B_n.
\]
Differentiate
\[
 B_{n(t)}P_\mu(t)=\mu P_\mu(t).
\]
At $t=0$ this gives
\[
 B_uP_\mu(n)+(B_n-\mu\Id)P_\mu'(0)=0.
\]
Multiplication on the left by $P_\mu(n)$ yields
\begin{equation}\label{eq:diagonal-block-zero}
 P_\mu(n) B_uP_\mu(n)=0.
\end{equation}
For $\nu\ne\mu$, multiplication on the left by $P_\nu(n)$, together with
the derivative of $P_\mu(n)^2=P_\mu(n)$, gives the full formula
\begin{equation}\label{eq:first-spectral-variation}
 P_\mu'(0)
 =\sum_{\nu\ne\mu}
 \frac{P_\nu(n) B_uP_\mu(n)+P_\mu(n) B_uP_\nu(n)}{\mu-\nu}.
\end{equation}

Differentiate the eigenprojection equation once more:
\[
 -B_nP_\mu(n)+2B_uP_\mu'(0)
 +(B_n-\mu\Id)P_\mu''(0)=0.
\]
Multiplying on the left and right by $P_\mu(n)$ and using
\eqref{eq:first-spectral-variation}, we obtain
\begin{equation}\label{eq:second-spectral-variation}
 \sum_{\nu\ne\mu}
 \frac{P_\mu(n) B_uP_\nu(n) B_uP_\mu(n)}{\mu-\nu}
 =\frac{\mu}{2}P_\mu(n).
\end{equation}
For an arbitrary $u\perp n$, the right-hand side is
$\frac{\mu}{2}|u|^2P_\mu(n)$ by homogeneity.

\smallskip
\noindent\textit{3). The Euler number estimate.}
Fix $n\in S^2$, and let $u_1,u_2$ be an oriented orthonormal basis of
$T_nS^2$.  For $\nu\ne a$, put
\[
 C_\nu(u)=P_\nu(n) B_uP_a(n):\mathcal E_a(n)\longrightarrow \mathcal E_\nu (n).
\]
Taking the trace of \eqref{eq:second-spectral-variation} at $\mu=a$
gives
\begin{equation}\label{eq:extreme-trace-identity}
 \sum_{\nu\ne a}
 \frac{\|C_\nu(u)\|_{\mathrm{HS}}^2}{a-\nu}
 =a|u|^2.
\end{equation}
By definition, for any orthonormal basis $e_1, e_2$ of $\mathcal E_a(n)$, the Hilbert--Schmidt norm of $C_\nu(u)$ is given by
$$\|C_\nu(u)\|_{\mathrm{HS}}^2:=|C_\nu(u)(e_1)|^2+|C_\nu(u)(e_2)|^2.$$

Equip $\mathcal E_a$ with the projected Euclidean connection
$\nabla^a=P_aD$, where $D$ is the flat connection on $S^2\times V$.
Recall that the curvature tensor $R^a$ of $\nabla^a$ is defined by
$$R^a(u, v):=\nabla^a_u\nabla^a_v-\nabla^a_v\nabla^a_u-\nabla^a_{[u,v]}.$$

Formula \eqref{eq:first-spectral-variation} gives its curvature as
\begin{equation}\label{eq:eigenbundle-curvature}
 R^a(u,v)=
 \sum_{\nu\ne a}
 \frac{C_\nu(u)^tC_\nu(v)-C_\nu(v)^tC_\nu(u)}
 {(a-\nu)^2}.
\end{equation}
Let  $e_1, e_2$ be an oriented orthonormal basis of $\mathcal E_a(n)$,
and define
\[
 \Omega_a(n):=\langle R^a(u_1,u_2)e_2, e_1\rangle.
\]
For any two real maps $S,T:E\to F$ from the same oriented Euclidean
plane $E$ to the same Euclidean space $F$, and for an oriented
orthonormal basis $f_1, f_2$ of $E$, the elementary inequality
\[
 |\langle Tf_1,Sf_2\rangle-\langle Sf_1,Tf_2\rangle|
 \le\frac12\bigl(\|S\|_{\mathrm{HS}}^2+\|T\|_{\mathrm{HS}}^2\bigr)
\]
follows by completing two squares.  Applying this to each term in
\eqref{eq:eigenbundle-curvature}, set
\[
 A(n):=\frac12\sum_{\nu\ne a}
 \frac{\|C_\nu(u_1)\|_{\mathrm{HS}}^2+
       \|C_\nu(u_2)\|_{\mathrm{HS}}^2}
 {(a-\nu)^2}.
\]
Then $|\Omega_a(n)|\le A(n)$.  Since $a-\nu\ge a-b$ and every
denominator is positive,
\[
 \frac1{(a-\nu)^2}
 \le\frac1{(a-b)(a-\nu)}.
\]
Using \eqref{eq:extreme-trace-identity} for $u_1$ and $u_2$, we obtain
\begin{equation}\label{eq:pointwise-euler-bound}
 |\Omega_a(n)|\le A(n)\le\frac{a}{a-b}=\frac32.
\end{equation}
By Chern-Weil theory \cite{MS74}, the real Euler curvature formula for an oriented Euclidean plane bundle
now gives
\begin{equation}\label{eq:euler-bound-chain}
 2\pi|d_a|
 =\left|\int_{S^2}\Omega_a\,dA\right|
 \le\int_{S^2}|\Omega_a|\,dA
 \le\frac32\operatorname{Area}(S^2)=6\pi.
\end{equation}
Thus $|d_a|\le3$, and the antipodal relation gives the same estimate for
$d_{-a}$.

The proof of part (ii) is complete.

\smallskip
\noindent\textit{4). What equality forces.}
Assume $|d_a|=3$.  Equality holds throughout
\eqref{eq:euler-bound-chain}.  The continuous function
$3/2-|\Omega_a|$ is nonnegative and has integral zero, so
$|\Omega_a|=3/2$ at every point.  The two inequalities in
\eqref{eq:pointwise-euler-bound} now give $A(n)=3/2$, so equality holds
in the spectral-gap estimate.  More explicitly,
\[
 \begin{split}
 0={}&\frac12\sum_{\rho=1}^2
 \sum_{\nu\in\{0,-b,-a\}}
 \left(
 \frac1{(a-b)(a-\nu)}-\frac1{(a-\nu)^2}
 \right)
 \|C_\nu(u_\rho)\|_{\mathrm{HS}}^2 .
 \end{split}
\]
Observe that every coefficient is strictly positive, hence
$C_\nu(u_\rho)=0$ for $\nu=0,-b,-a$.  Since the tangent basis was
arbitrary, for every $n\in S^2$ and every $u\perp n$,
\begin{equation}\label{eq:extreme-block-vanishing}
 B_u\mathcal E_a(n)\subset \mathcal E_b(n).
\end{equation}
Equation \eqref{eq:second-spectral-variation} now reduces to
\[
 \frac{(P_b(n)B_uP_a(n))^t(P_b(n)B_uP_a(n))}{a-b}
 =\frac a2|u|^2\Id.
\]
Because $a(a-b)/2=1$,
\begin{equation}\label{eq:outer-block-isometry}
 (P_b(n)B_uP_a(n))^t(P_b(n)B_uP_a(n))=|u|^2\Id.
\end{equation}
Applying these conclusions at $-n$, and recalling that
$\mathcal E_a(-n)=\mathcal E_{-a}(n)$ and $\mathcal E_b(-n)=\mathcal E_{-b}(n)$, gives
\begin{equation}\label{eq:negative-extreme-block-vanishing}
 B_u\mathcal E_{-a}(n)\subset \mathcal E_{-b}(n),
\end{equation}
and the corresponding restriction is again a scaled isometry.

\smallskip
\noindent\textit{5). The zero eigenspace is a common kernel.}
Fix orthonormal $n,u\in\R^3$ and use the same great circle $n(t)$.
Its unit tangent vector is
\[
 m(t)=n'(t)=-\sin t\,n+\cos t\,u.
\]
By \eqref{eq:extreme-block-vanishing} and the symmetry of $B_{m(t)}$,
\[
 P_a(t)B_{m(t)}P_0(t)=0
\]
for every $t$.  Differentiating at $t=0$ gives
\[
 P_a'(0)B_uP_0(n)-P_a(n)B_nP_0(n)+P_a(n)B_uP_0'(0)=0.
\]
The middle term is zero.  Take the block from $E_0(n)$ to $E_a(n)$ and
put
\[
 U=P_a(n)B_uP_b(n),\qquad V_0=P_b(n)B_uP_0(n).
\]
Formula \eqref{eq:first-spectral-variation} and the extreme block
vanishings give
\[
 P_a(n)P_a'(0)B_uP_0(n)=\frac{UV_0}{a-b},
 \qquad
 P_a(n)B_uP_0'(0)P_0(n)=-\frac{UV_0}{b}.
\]
Consequently,
\[
 \left(\frac1{a-b}-\frac1b\right)UV_0
 =-\frac1{2b}UV_0=0.
\]
By \eqref{eq:outer-block-isometry}, $U$ is invertible for $u\ne0$.
Thus $P_b(n)B_uP_0(n)=0$.  Repeating the argument at $-n$ gives
$P_{-b}(n)B_uP_0(n)=0$.  The remaining components vanish by
\eqref{eq:diagonal-block-zero},
\eqref{eq:extreme-block-vanishing},
\eqref{eq:negative-extreme-block-vanishing}, and symmetry.  Hence
\[
 B_uE_0(n)=0\qquad(u\perp n).
\]
Together with $B_nE_0(n)=0$, linearity gives
$E_0(n)\subset\bigcap_{\xi\in\R^3}\ker B_\xi$.  The reverse inclusion
holds because every vector in the intersection lies in
$\ker B_n=E_0(n)$.  Therefore
\begin{equation}\label{eq:common-kernel}
 K:=\bigcap_{\xi\in\R^3}\ker B_\xi=E_0(n),
 \qquad \dim K=2,
\end{equation}
for every $n\in S^2$.

\smallskip
\noindent\textit{6). The three real off-diagonal blocks.}
Every $B_\xi$ preserves $K^\perp$.  Choose an orthonormal basis
$n_1,n_2,n_3$ of $\R^3$ and decompose
\[
 K^\perp=W_a\oplus W_b\oplus W_{-b}\oplus W_{-a}
\]
into the eigenspaces of $B_{n_3}$.  Thus
\[
 B_{n_3}|_{K^\perp}
 =\operatorname{diag}(aI,bI,-bI,-aI).
\]
For $u\in n_3^\perp$, the diagonal identity, the common kernel, and the
two extreme block vanishings leave only three successive blocks:
\begin{equation}\label{eq:tridiagonal-pencil}
 B_u|_{K^\perp}=
 \begin{pmatrix}
 0&U(u)&0&0\\
 U(u)^t&0&M(u)&0\\
 0&M(u)^t&0&Z(u)\\
 0&0&Z(u)^t&0
 \end{pmatrix}.
\end{equation}
Here $U(u):W_b\to W_a$, $M(u):W_{-b}\to W_b$, and
$Z(u):W_{-a}\to W_{-b}$.  The outer isometry relation and its
antipodal counterpart first give
\[
 U(u)U(u)^t=|u|^2I,\qquad Z(u)^tZ(u)=|u|^2I.
\]
These are square $2$-by-$2$ blocks.  Thus, when $u\ne0$, each block is
invertible, and the same identities also hold with the transpose on the
other side.  They are trivial for $u=0$.
Apply \eqref{eq:second-spectral-variation} at the eigenvalue $b$.  The
only two terms come from $W_a$ and $W_{-b}$:
\[
 \frac{U(u)^tU(u)}{b-a}
 +\frac{M(u)M(u)^t}{2b}
 =\frac b2|u|^2I.
\]
Since $a-b=2b$ and $b^2=1/3$, this gives
\[
 M(u)M(u)^t=\frac43|u|^2I.
\]
Therefore,
\begin{equation}\label{eq:three-block-norms}
 U(u)^tU(u)=|u|^2I,\qquad
 M(u)^tM(u)=c^2|u|^2I,\qquad
 Z(u)^tZ(u)=|u|^2I.
\end{equation}

\smallskip
\noindent\textit{7). Simultaneous normalization.}
The three blocks at $n_1$ must be normalized successively, not by three
independent choices.  Choose an orthonormal basis
$e_{a,1},e_{a,2}$ of $W_a$, and define
\[
 e_{b,j}=U(n_1)^te_{a,j},\qquad
 e_{-b,j}=\frac1cM(n_1)^te_{b,j},\qquad
 e_{-a,j}=Z(n_1)^te_{-b,j}.
\]
Equation \eqref{eq:three-block-norms} shows at each step that the new
basis is orthonormal.  In these fixed bases,
\begin{equation}\label{eq:first-tangent-normal-form}
 U(n_1)=I,\qquad M(n_1)=cI,\qquad Z(n_1)=I.
\end{equation}
All basis changes take place inside eigenspaces of $B_{n_3}$, so they
do not alter its diagonal form.

Each block depends linearly on $u$.  Polarizing
\eqref{eq:three-block-norms} for $u=xn_1+yn_2$ shows that
$U(n_2)$, $M(n_2)/c$, and $Z(n_2)$ are skew-symmetric orthogonal
$2$-by-$2$ matrices.  Hence
\begin{equation}\label{eq:three-signs}
 U(n_2)=\varepsilon_1J,\qquad
 M(n_2)=c\varepsilon_2J,\qquad
 Z(n_2)=\varepsilon_3J,
 \qquad \varepsilon_j\in\{1,-1\}.
\end{equation}

The three signs are forced to agree.  Consider
\[
 \gamma(t)=\cos t\,n_3+\sin t\,n_1.
\]
The vector $n_2$ is perpendicular to $\gamma(t)$ for every $t$.
The positive-extreme vanishing, followed by symmetry, gives the first
identity below; the negative-extreme vanishing gives the second:
\[
 P_a(\gamma(t))B_{n_2}P_{-b}(\gamma(t))=0,
 \qquad
 P_b(\gamma(t))B_{n_2}P_{-a}(\gamma(t))=0.
\]
Differentiate them at $t=0$ and use both terms in
\eqref{eq:first-spectral-variation}.  Taking the block from
$E_{-b}(n_3)$ to $E_a(n_3)$ in the first differentiated identity gives
\[
 \frac{U(n_1)M(n_2)}{a-b}
 -\frac{U(n_2)M(n_1)}{2b}=0,
\]
Taking the block from $E_{-a}(n_3)$ to $E_b(n_3)$ in the second gives
\[
 \frac{M(n_1)Z(n_2)}{2b}
 -\frac{M(n_2)Z(n_1)}{a-b}=0.
\]
Since $a-b=2b$, we obtain
\[
 U(n_1)M(n_2)=U(n_2)M(n_1),\qquad
 M(n_1)Z(n_2)=M(n_2)Z(n_1).
\]
Substitution of \eqref{eq:first-tangent-normal-form} and
\eqref{eq:three-signs} gives
\[
 \varepsilon_1=\varepsilon_2=\varepsilon_3.
\]
Replacing $n_2$ by $-n_2$ if necessary makes the common sign equal to
$1$.  Combining the matrices of $B_{n_1},B_{n_2},B_{n_3}$ now gives
\eqref{eq:canonical-extremal-pencil}.  This differentiation is
essential: the norm identities alone leave three independent signs.

Now, the proof is complete.
\end{proof}

\begin{cor}\label{cor:focal-pencil-rigidity}
After interchanging $N_+$ and $N_-$ as in
Proposition~\ref{prop:focal-euler-numbers}, the normal shape pencil of
$N_+$ at every point has the real canonical form
\eqref{eq:canonical-extremal-pencil}.  In particular, given any $p\in N_+$,
the kernel of $B_n$ is independent of $n\in \nu_p^{S^{13}}N_+$.
\end{cor}
\begin{proof}
For every $p\in N_+$, Proposition~\ref{prop:focal-euler-numbers} gives
$|n_{\sqrt3}(p)|=3$.  The conclusion follows from
Theorem~\ref{thm:extremal-pencil}.
\end{proof}
\begin{rem}
A substantial portion of Miyaoka's Annals papers \cite{Mi13, Mi16} is devoted to the proof of Corollary 3.1.
\end{rem}
%----------------------------------------------------------------------------------------------------------------
\section{Existence of Ambrose-Singer connection and a Proof of Theorem 1.1}
\subsection{The Codazzi cubic of the hypersurface}
As before, let $M$ be an isoparamtric hypersurface in $S^{13}$ with $(g, m)=(6, 2)$, $\xi$ be a
globally defined unit normal in the sphere, $A$ be the shape operator with respect to $\xi$, and $\partial$ be
the flat connection of $\R^{14}$. We use the shape-operator convention
\[
 \partial_X\xi=-AX.
\]
Write
\begin{equation}\label{ck:eq-principal}
 TM=D_1\oplus\cdots\oplus D_6,
 \qquad A|_{D_i}=\lambda_i\Id,
 \qquad \operatorname{rank}D_i=2.
\end{equation}

Let $\langle\,,\,\rangle$ be the induced metric of $M$ in $S^{13}$ and set
\begin{equation}\label{ck:eq-cubic}
 h(Y,Z)=\langle AY,Z\rangle,
 \qquad
 \alpha(X,Y,Z)=(\nabla_Xh)(Y,Z)
             =\langle(\nabla_XA)Y,Z\rangle.
\end{equation}
It follows from the Codazzi equation
that $\alpha$ is a symmetric covariant cubic tensor. The following basic lemma is well known and we therefore omit the proof.

\begin{lem}\label{ck:lem-cubic}
For sections $Y$ of $D_j$ and $Z$ of $D_k$, and an arbitrary tangent
vector field $X$ on $M$, one has
\begin{equation}\label{ck:eq-cubic-components}
 \alpha(X,Y,Z)=(\lambda_j-\lambda_k)
                    \langle\nabla_XY,Z\rangle.
\end{equation}
In particular, $\alpha$ vanishes whenever two of its arguments belong
to the same principal distribution.
\end{lem}

\subsection{The common kernel and opposite principal distributions}
For $1\le i\le6$, define the focal map and its associated unit normal by
\begin{equation}\label{ck:eq-focal-maps}
 \begin{split}
 F_i(x)&=\cos\theta_i\,x+\sin\theta_i\,\xi(x),\\
 \eta_i(x)&=-\sin\theta_i\,x+\cos\theta_i\,\xi(x).
 \end{split}
\end{equation}
The images $N_i=F_i(M)$ alternate between the two focal submanifolds $N_{\pm}$.
For completeness, if $\varphi$ is the restriction of a compatibly
normalized Cartan--M\"unzner polynomial, its normal-circle formula is
\cite{Mu80}
\begin{equation}\label{ck:eq-normal-circle}
 \varphi(\cos t\,x+\sin t\,\xi(x))
 =\cos\bigl(6(\theta_1-t)\bigr).
\end{equation}
Thus
\[
 \varphi(F_i(x))=(-1)^{i-1}.
\]
Let $N$ be the focal submanifold $N_+$ in Corollary
\ref{cor:focal-pencil-rigidity}. The index set
\[
 \mathcal I=\{i\in\{1,\ldots,6\}:N_i=N\}
\]
is therefore either $\{1,3,5\}$ or $\{2,4,6\}$.
Every opposite pair $\{a,a+3\}$, $1\le a\le3$, contains exactly one
element of $\mathcal I$.  This observation avoids imposing the
common-kernel condition on both focal submanifolds.

\begin{lem}\label{ck:lem-opposite}
Assume for any given $p\in N$, the kernel of the shape operator $B_n$ of $N$ is independent of $n\in \nu_p^{S^{13}}N$. Then
\begin{equation}\label{ck:eq-opposite}
 \alpha(D_a,D_{a+3},TM)=0,
 \qquad a=1,2,3.
\end{equation}
\end{lem}
\begin{proof}
Fix $i\in\mathcal I$, $x\in M$, and $q=F_i(x)$.  For
$X\in D_j(x)$, differentiating \eqref{ck:eq-focal-maps} gives
\begin{align}
 dF_i(X)
 &=\bigl(\cos\theta_i-\lambda_j\sin\theta_i\bigr)X
   =\frac{\sin(\theta_j-\theta_i)}{\sin\theta_j}X,
   \label{ck:eq-df}\\
 d\eta_i(X)
 &=-\bigl(\sin\theta_i+\lambda_j\cos\theta_i\bigr)X
   =-\frac{\cos(\theta_j-\theta_i)}{\sin\theta_j}X.
   \label{ck:eq-deta}
\end{align}
The coefficient in \eqref{ck:eq-df} is zero precisely when $j=i$.
Hence $\ker dF_i=D_i$ and $\operatorname{rank}dF_i=10$.
Then, we obtain the orthogonal direct sum
\begin{equation}\label{ck:eq-focal-tangent}
 T_qN=\bigoplus_{j\ne i}D_j(x).
\end{equation}
The vector $\eta_i(x)$ is perpendicular to both $q$ and the right-hand
side of \eqref{ck:eq-focal-tangent}. It is therefore a unit vector in
$\nu_q^{S^{13}}N$, a three-dimensional normal space.

When $j\ne i$, the vector in \eqref{ck:eq-deta} belongs to $T_qN$.
The Weingarten formula along the map $F_i$ consequently yields
\begin{equation}\label{ck:eq-focal-shape}
 B_{\eta_i(x)}dF_i(X)
 =-d\eta_i(X)
 =\cot(\theta_j-\theta_i)\,dF_i(X).
\end{equation}
%This formula uses the tangential component of the derivative of the
%normal field along a curve. It does not require $\eta_i$ to descend
%to a single-valued normal field on $N$.
It follows that the five eigenvalues of $B_{\eta_i(x)}$ are
\[
 \sqrt3,\quad\frac1{\sqrt3},\quad0,
 \quad-\frac1{\sqrt3},\quad-\sqrt3,
\]
each with multiplicity two.  The zero eigenvalue corresponds to
$j\equiv i+3\pmod6$.  Thus
\begin{equation}\label{ck:eq-focal-kernel}
 \ker B_{\eta_i(x)}=D_{i+3}(x).
\end{equation}

Now keep $q$ fixed and vary $x$ along a fiber of $F_i$.
The fact that for any $n\in \nu_q^{S^{13}}N$, the kernel of the shape operator $B_n$ of $N$ is independent of $n$, and equation \eqref{ck:eq-focal-kernel} imply
\begin{equation}\label{ck:eq-kernel-along-fiber}
 D_{i+3}(x)=K_q
 \qquad\text{for all }x\in F_i^{-1}(q),
\end{equation}
where $K_q$ is a fixed subspace.
Let $X$ be a local section of $D_i$ and $Y$ a local section of
$D_{i+3}$.  Along each integral curve of $X$, the map $F_i$ is constant.
By \eqref{ck:eq-kernel-along-fiber}, we have $Y\in K_q$, and therefore
\[
 \partial_XY\in K_q.
\]
On the other hand, the Euclidean Gauss formula is
\begin{equation}\label{ck:eq-euclidean-gauss}
 \partial_XY=\nabla_XY+h(X,Y)\xi-\langle X,Y\rangle x.
\end{equation}
The last two terms vanish because $X$ and $Y$ belong to distinct
principal distributions.  Consequently,
\[
 \nabla_XY\in K_q=D_{i+3}.
\]
Equation \eqref{ck:eq-cubic-components} now gives
$\alpha(D_i,D_{i+3},TM)=0$.
As $i$ runs through $\mathcal I$, the symmetry of $\alpha$ proves \eqref{ck:eq-opposite}.
\end{proof}

\begin{rem}
See also Theorem 4.7 in \cite{Si17} for an alternative proof of Lemma \ref{ck:lem-opposite}.
\end{rem}

\subsection{Construction of a metric connection}
Set
\begin{equation}\label{ck:eq-three-bundles}
 \mathcal H_1=D_1\oplus D_4,
 \qquad
 \mathcal H_2=D_2\oplus D_5,
 \qquad
 \mathcal H_3=D_3\oplus D_6.
\end{equation}
These are mutually orthogonal smooth subbundles of rank four.
By Lemmas~\ref{ck:lem-cubic} and~\ref{ck:lem-opposite},
\begin{equation}\label{ck:eq-three-vanishing}
 \alpha(U,V,W)=0
 \quad\text{if two arguments belong to the same }\mathcal H_a.
\end{equation}
Thus the only potentially nonzero components of $\alpha$ have one
argument in each of the three bundles.

Let $\Pi_a$ be the orthogonal projection onto $\mathcal H_a$ and define
\begin{equation}\label{ck:eq-connection}
 \nabla^c_XY=\sum_{a=1}^{3}\Pi_a\nabla_X(\Pi_aY),
 \qquad
 \Gamma_XY=\nabla_XY-\nabla^c_XY.
\end{equation}
It is clear that $\nabla^c$ is a globally defined smooth
connection, and $\Gamma$ is a tensor of type $(1,2)$.

\begin{prop}\label{ck:prop-connection}
The connection $\nabla^c$ is metric and preserves each $D_i$.
Moreover,
\begin{equation}\label{ck:eq-parallel-shape}
 \nabla^c A=0,
 \qquad \nabla^c h=0.
\end{equation}
The difference tensor satisfies
\begin{align}
 \Gamma_{\mathcal H_a}\mathcal H_a&=0,
                    \label{ck:eq-gamma-same}\\
 \Gamma_{\mathcal H_a}\mathcal H_b&\subset\mathcal H_c
 \quad\text{if }\{a,b,c\}=\{1,2,3\}.
                    \label{ck:eq-gamma-mixed}
\end{align}
For $Y\in D_j$ and $Z\in D_k$, its components are
\begin{equation}\label{ck:eq-gamma-components}
 \langle\Gamma_XY,Z\rangle=
 \begin{cases}
 \displaystyle\frac{\alpha(X,Y,Z)}{\lambda_j-\lambda_k},&j\ne k,\\[6pt]
 0,&j=k.
 \end{cases}
\end{equation}
\end{prop}
\begin{proof}
Write $Y_a=\Pi_aY$ and $Z_a=\Pi_aZ$.  Orthogonality and metric
compatibility of $\nabla$ give
\begin{align*}
 X\langle Y,Z\rangle
 &=\sum_a\bigl(\langle\nabla_XY_a,Z_a\rangle
                  +\langle Y_a,\nabla_XZ_a\rangle\bigr)\\
 &=\langle\nabla^c_XY,Z\rangle
                  +\langle Y,\nabla^c_XZ\rangle.
\end{align*}
Hence $\nabla^c\langle\,,\,\rangle=0$.
By construction, $\nabla^c$ preserves each $\mathcal H_a$.

If $Y\in D_i$ and $Z\in D_{i+3}$, then
\eqref{ck:eq-cubic-components} and \eqref{ck:eq-opposite} yield
\[
 (\lambda_i-\lambda_{i+3})\langle\nabla_XY,Z\rangle=0.
\]
It follows that
$\nabla^c_XY\in D_i$ whenever $Y\in D_i$.
Since the $\lambda_i$ are constant, $\nabla^c A=0$. As $\nabla^c$ is metric, it follows that $\nabla^c h=0$.

For $Y\in D_j$, $Z\in D_k$ and $j\ne k$, since $\nabla^c$ preserves $D_j$, it implies that
$\langle\nabla^c_XY,Z\rangle=0$.  Equation
\eqref{ck:eq-cubic-components} proves the first case of
\eqref{ck:eq-gamma-components}.  For $j=k$, projection onto the
bundle containing $D_j$ leaves the inner product with $Z$ unchanged;
thus $\langle\Gamma_XY,Z\rangle=0$.

Suppose next that $X,Y\in\mathcal H_a$.  Decompose $Y$ into its two
principal components. For any principal component of a vector
orthogonal to $\mathcal H_a$, equations
\eqref{ck:eq-three-vanishing} and \eqref{ck:eq-cubic-components}
show that its inner product with $\nabla_XY$ is zero.
Hence $\nabla_XY\in\mathcal H_a$, and
\eqref{ck:eq-connection} gives $\Gamma_XY=0$.
This proves \eqref{ck:eq-gamma-same}.

Finally, let $X\in\mathcal H_a$, $Y\in\mathcal H_b$, and $a\ne b$.
The same component calculation, now testing against
$Z\in\mathcal H_a$, eliminates the $\mathcal H_a$-component of
$\nabla_XY$, since the first and third arguments of $\alpha(X,Y,Z)$
lie in $\mathcal H_a$.  The $\mathcal H_b$-component is precisely
$\nabla^c_XY$.  Their difference therefore lies in the remaining
bundle $\mathcal H_c$, proving \eqref{ck:eq-gamma-mixed}.
\end{proof}

\subsection{Parallelism of the cubic tensor and the difference tensor}
\begin{prop}\label{ck:prop-parallelism}
\begin{equation}\label{ck:eq-parallel-cubic-gamma}
 \nabla^c\alpha=0,
 \qquad \nabla^c\Gamma=0.
\end{equation}
\end{prop}
\begin{proof}
Since $\nabla^c$ preserves each $\mathcal H_a$, differentiating
\eqref{ck:eq-three-vanishing} with this connection gives
\begin{equation}\label{ck:eq-differentiated-vanishing}
 (\nabla^c_X\alpha)(U,V,W)=0
\end{equation}
whenever two of $U,V,W$ lie in the same $\mathcal H_a$.
It remains to consider one argument in each bundle.  The tensor
$\nabla^c_X\alpha$ is symmetric in its three cubic arguments.
Decomposing the derivative argument among the three bundles, and
then using this symmetry, reduces the proof to
\begin{equation}\label{ck:eq-reduction}
 (\nabla^c_X\alpha)(Z,Y,W)=0,
 \qquad
 X,Y\in\mathcal H_a,\quad Z\in\mathcal H_b,\quad W\in\mathcal H_c,
\end{equation}
where $\{a,b,c\}=\{1,2,3\}$.

We use the curvature convention
\[
 R(X,Z)=\nabla_X\nabla_Z-\nabla_Z\nabla_X-\nabla_{[X,Z]}.
\]
Because $\alpha=\nabla h$, the Ricci identity reads
\begin{equation}\label{ck:eq-ricci}
 \begin{split}
 (\nabla_X\alpha)(Z,Y,W)-(\nabla_Z\alpha)(X,Y,W)
 &=-h(R(X,Z)Y,W)\\
 &\phantom{={}}-h(Y,R(X,Z)W).
 \end{split}
\end{equation}
The Gauss equation of $M\subset S^{13}$ is
\begin{equation}\label{ck:eq-gauss}
 \begin{split}
 R(X,Z)Y={}&\langle Z,Y\rangle X-\langle X,Y\rangle Z\\
           &+h(Z,Y)AX-h(X,Y)AZ.
 \end{split}
\end{equation}
For the vectors in \eqref{ck:eq-reduction}, orthogonality and
$A$-invariance of the $\mathcal H_a$ imply
\[
 R(X,Z)Y=-\langle X,Y\rangle Z-h(X,Y)AZ\in\mathcal H_b,
 \qquad R(X,Z)W=0.
\]
The right-hand side of \eqref{ck:eq-ricci} therefore vanishes, and
\begin{equation}\label{ck:eq-ricci-zero}
 (\nabla_X\alpha)(Z,Y,W)=(\nabla_Z\alpha)(X,Y,W).
\end{equation}

For any covariant cubic tensor, the comparison of $\nabla$ and
$\nabla^c$ gives
\begin{equation}\label{ck:eq-cubic-comparison}
 \begin{split}
 (\nabla_X\alpha)(Z,Y,W)
 ={}&(\nabla^c_X\alpha)(Z,Y,W)
       -\alpha(\Gamma_XZ,Y,W)\\
    &-\alpha(Z,\Gamma_XY,W)
       -\alpha(Z,Y,\Gamma_XW).
 \end{split}
\end{equation}
By \eqref{ck:eq-gamma-same}--\eqref{ck:eq-gamma-mixed},
\[
 \Gamma_XY=0,
 \qquad \Gamma_XZ\in\mathcal H_c,
 \qquad \Gamma_XW\in\mathcal H_b.
\]
Each correction term in \eqref{ck:eq-cubic-comparison} vanishes:
the first has two arguments in $\mathcal H_c$, the middle term is
zero, and the last has two arguments in $\mathcal H_b$.
Consequently,
\begin{equation}\label{ck:eq-first-ricci-term}
 (\nabla_X\alpha)(Z,Y,W)=(\nabla^c_X\alpha)(Z,Y,W).
\end{equation}

For the other side of \eqref{ck:eq-ricci-zero}, equation
\eqref{ck:eq-differentiated-vanishing} gives
\[
 (\nabla^c_Z\alpha)(X,Y,W)=0.
\]
Moreover,
\[
 \Gamma_ZX,\Gamma_ZY\in\mathcal H_c,
 \qquad \Gamma_ZW\in\mathcal H_a.
\]
Thus
\[
 \alpha(\Gamma_ZX,Y,W)
 =\alpha(X,\Gamma_ZY,W)
 =\alpha(X,Y,\Gamma_ZW)=0,
\]
by \eqref{ck:eq-three-vanishing}.  Applying the same
connection-comparison formula yields
\begin{equation}\label{ck:eq-second-ricci-term}
 (\nabla_Z\alpha)(X,Y,W)=0.
\end{equation}
Combining \eqref{ck:eq-ricci-zero},
\eqref{ck:eq-first-ricci-term}, and
\eqref{ck:eq-second-ricci-term} proves \eqref{ck:eq-reduction}.
This establishes $\nabla^c\alpha=0$.

It remains to pass from the cubic tensor to the difference tensor.
The metric and each $D_j$ are $\nabla^c$-parallel, and all
$\lambda_j$ are constant.  Differentiating
\eqref{ck:eq-gamma-components}, for $Y\in D_j$, $Z\in D_k$, and
$j\ne k$, gives
\[
 \bigl\langle(\nabla^c_U\Gamma)(X,Y),Z\bigr\rangle
 =\frac{(\nabla^c_U\alpha)(X,Y,Z)}{\lambda_j-\lambda_k}=0.
\]
When $j=k$, the corresponding component is zero by differentiating
the second case of \eqref{ck:eq-gamma-components}.
These are all components of $\nabla^c\Gamma$, so
$\nabla^c\Gamma=0$.
\end{proof}

\subsection{Parallel torsion and parallel curvature}
Recall that by definition a metric connection with
parallel-torsion and parallel-curvature conditions is said to be an \emph{Ambrose--Singer connection}.
\begin{cor}\label{ck:cor-as}
The torsion $T^c$
and curvature $R^c$ of $\nabla^c$ satisfy
\[
 \nabla^cT^c=0,
 \qquad\nabla^cR^c=0.
\]
Thus $\nabla^c$ is an Ambrose--Singer connection.
\end{cor}
\begin{proof}
Since $\nabla=\nabla^c+\Gamma$ is torsion-free,
\begin{equation}\label{ck:eq-torsion}
 T^c(X,Y)=-\Gamma_XY+\Gamma_YX.
\end{equation}
Hence $\nabla^cT^c=0$.  The Gauss equation
\eqref{ck:eq-gauss} expresses $R$ algebraically in the metric and $A$;
therefore $\nabla^cR=0$.  With $\nabla^c\Gamma=0$, the
curvature-comparison formula becomes
\begin{equation}\label{ck:eq-curvature-comparison}
 R(X,Y)=R^c(X,Y)+[\Gamma_X,\Gamma_Y]+\Gamma_{T^c(X,Y)}.
\end{equation}
Every term on the right other than $R^c$ is $\nabla^c$-parallel.
Consequently, $\nabla^cR^c=0$.
\end{proof}

\subsection{A proof of Theorem 1.1}
\begin{proof}
Let $M$ be a closed isoparametric hypersurface in $S^{13}$ with $(g, m)=(6, 2)$. By Corollary 3.1,
Propositions 4.1, 4.2 and Corollary 4.1, there exists an Ambrose-Singer connection $\nabla^c$ on $M$. According to \cite{AS58}, $M$ is intrinsically homogeneous.
Since $M$ is closed and simply connected, by the rigidity theorem of hypersurfaces with type
number larger than two \cite{KN69}, it follows that $M$ must be extrinsically
homogeneous (See also P. 108 of \cite{Mi13}). This completes the proof.
\end{proof}

\begin{center}\footnotesize{ACKNOWLEDGMENTS}\end{center}
The authors are grateful to Professors Thomas E. Cecil, Quo-Shin Chi, Reiko Miyaoka, and Gudlaugur Thorbergsson for their valuable comments and suggestions.

\section*{AI Assistance Disclosure}

The AI tool ChatGPT (OpenAI) was used in the preparation of this manuscript. The authors have independently verified all AI-assisted content and take
full responsibility for the final manuscript.

\end{document}